\documentclass[11pt, reqno]{amsart}
\usepackage{amsmath, amsthm, amsfonts, amssymb, amsbsy, bigstrut, graphicx, enumerate,  upref, longtable, comment, booktabs, array, caption, subcaption}

\usepackage{pstricks}
\usepackage{times}

\usepackage[numbers, sort&compress]{natbib}

\usepackage[breaklinks]{hyperref}

\newtheorem{thm}{Theorem}[section]
\newtheorem{lmm}[thm]{Lemma}
\newtheorem{cor}[thm]{Corollary}

\newtheorem{conj}[thm]{Conjecture}
\theoremstyle{definition}
\newtheorem{remark}[thm]{Remark}

\newcommand{\ee}{\mathbb{E}}

\newcommand{\mf}{\mathcal{F}}

\newcommand{\cp}{\mathcal{P}}
\newcommand{\cq}{\mathcal{Q}}

\newcommand{\pp}{\mathbb{P}}

\newcommand{\rr}{\mathbb{R}}

\newcommand{\ve}{\varepsilon}

\newcommand{\zz}{\mathbb{Z}}

\numberwithin{equation}{section}

\usepackage[utf8]{inputenc}
\usepackage{amsmath}
\usepackage{amsfonts}
\usepackage{bbm}
\usepackage{mathtools}
\usepackage{tikz}
\usetikzlibrary{decorations.pathreplacing}
\usepackage{amsthm}
\usepackage[english]{babel}
\usepackage{fancyhdr}
\usepackage{amssymb}
\usepackage{enumitem}
\usepackage{qtree}
\usepackage{mathrsfs}

\begin{document}
\title[Stationary ballistic deposition]{Existence of stationary ballistic deposition on the infinite lattice}
\author{Sourav Chatterjee}

\address{Departments of mathematics and statistics, Stanford University}
\email{souravc@stanford.edu}
\thanks{Research partially supported by NSF grants DMS-1855484 and DMS-2113242}

\begin{abstract}
Ballistic deposition is one of the many models of interface growth that are believed to be in the KPZ universality class, but have so far proved to be largely intractable mathematically. In this model, blocks of size one fall independently as Poisson processes  at each site on the $d$-dimensional lattice, and either attach themselves to the column growing at that site, or to the side of an adjacent column, whichever comes first. It is not hard to see that if we subtract off the height of the column at the origin from the heights of the other columns, the resulting interface process is Markovian. The main result of this article is that this Markov process has at least one invariant probability measure. We conjecture that the invariant measure is not unique, and provide some partial evidence. 
\end{abstract}

\maketitle

\section{Introduction}
\subsection{The model}
Ballistic deposition is a popular model of interface (or surface) growth introduced by~\citet{vold59} and subsequently studied by many authors. The standard version of the model considered in the probability literature~(e.g., in~\cite{seppalainen00, penrose08, penroseyukich02}) is the following. In the $(d+1)$-dimensional version of the model, the surface at time $t\in [0,\infty)$ is described by a height function $h(t,\cdot)$, with $h(t,x)$ denoting the height of the surface at location $x\in \zz^d$. At each location, there is a Poisson clock that rings at rate~$1$. Each time the clock rings at $x$, a block of size $1$ falls at $x$ from infinity. As the block falls, it either attaches itself to the top of the existing column of blocks at $x$, or to the side of one of the neighboring columns, whichever comes first (see Figure \ref{illusfig} for an illustration). That is, if the clock rings at time $x$, the height at $x$ gets updated instantaneously to 
\[
\max\biggl\{\max_{1\le i\le d} h(t, x\pm e_i), \, h(t,x)+1\biggr\},
\]
where $e_1,\ldots,e_d$ are the standard basis vectors of $\rr^d$. This becomes the value of $h(s,x)$ for all $s\in (t,t']$, where $t'$ is the next time that the clock rings at site $x$. It is not obvious that this process is well-defined on the infinite lattice $\zz^d$, because the updates happen in continuous time instead of discrete, and they are happening all over the lattice. For a proof of the fact that $h(t,x)$ is well-defined and finite almost surely, see~\cite{penrose08}.

\begin{figure}
\centering
\begin{tikzpicture}[scale = .5]
\draw [fill = lightgray] (0,0) rectangle (4,1);
\draw [fill = lightgray] (5,0) rectangle (16,1);
\draw [fill = lightgray] (18,0) rectangle (20,1);
\draw (1,0) -- (1,1);
\draw (2,0) -- (2,1);
\draw (3,0) -- (3,1);
\draw (6,0) -- (6,1);
\draw (7,0) -- (7,1);
\draw (8,0) -- (8,1);
\draw (9,0) -- (9,1);
\draw (10,0) -- (10,1);
\draw (11,0) -- (11,1);
\draw (12,0) -- (12,1);
\draw (13,0) -- (13,1);
\draw (14,0) -- (14,1);
\draw (15,0) -- (15,1);
\draw (19,0) -- (19,1);
\draw [fill = lightgray] (1,1) rectangle (3,2);
\draw [fill = lightgray] (5,1) rectangle (6,2);
\draw [fill = lightgray] (9,1) rectangle (10,3);
\draw [fill = lightgray] (11,1) rectangle (14,2);
\draw [fill = lightgray] (15,1) rectangle (16,3);
\draw [fill = lightgray] (18,1) rectangle (20,2);
\draw [fill = lightgray] (2,2) rectangle (3,6);
\draw [fill = lightgray] (5,2) rectangle (7,3);
\draw [fill = lightgray] (13,2) rectangle (14,4);
\draw [fill = lightgray] (18,2) rectangle (19,4);
\draw [fill = lightgray] (17,3) rectangle (18,4);
\draw [fill = lightgray] (9,3) rectangle (11,4);
\draw [fill = lightgray] (9,4) rectangle (10,5);
\draw (10,3) -- (10,4);
\draw (2,1) -- (2,2);
\draw (2,3) -- (3,3);
\draw (2,4) -- (3,4);
\draw (2,5) -- (3,5);
\draw (6,2) -- (6,3);
\draw (9,2) -- (10,2);
\draw (12,1) -- (12,2);
\draw (13,1) -- (13,2);
\draw (13,3) -- (14,3);
\draw (15,2) -- (16,2);
\draw (18,3) -- (19,3);
\draw (4,0) -- (5,0);
\draw (16,0) -- (18,0);
\draw (19,1) -- (19,2);
\draw (4,10) rectangle (5,11);
\draw (9,11) rectangle (10,12);
\draw (12,12) rectangle (13,13);
\draw [dashed] (4,2) rectangle (5,3);
\draw [dashed] (9,5) rectangle (10,6);
\draw [dashed] (12,3) rectangle (13,4);
\draw [-stealth] (4.5,9) -- (4.5,7);
\draw [-stealth] (9.5,10) -- (9.5,8);
\draw [-stealth] (12.5,11) -- (12.5,9);
\end{tikzpicture}
\caption{Example of ballistic deposition growth: The three blocks falling from above attach themselves either to top of a column, or to the side of a column, whichever comes first. The dashed rectangles denote their final locations after dropping.\label{illusfig}}
\end{figure}
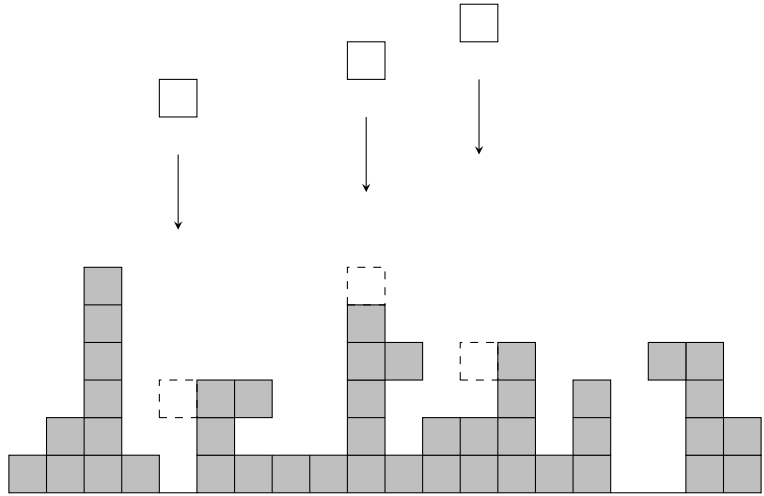

\subsection{Results}
It is clear that $\{h(t,\cdot)\}_{t\ge 0}$ is a time-homogeneous Markov process on the state space $\Omega := \zz^{\zz^d}$, 
equipped with the product topology and the Borel $\sigma$-algebra generated by this topology. One cannot expect that this process has a stationary law, since the height at any site will tend to infinity as $t\to\infty$. However, it turns out that if we subtract off $h(t,0)$ from every $h(t,x)$, then the resulting process does have a stationary distribution. This is the main result of this article.
\begin{thm}\label{mainthm}
Let $w(t,x) := h(t,x)-h(t,0)$. The process $\{w(t,\cdot)\}_{t\ge 0}$ is a time-homogeneous Markov process on $\Omega$ with at least one stationary probability distribution. The stationary distribution can be chosen such that the height at each location has finite $L^1$ norm. Equivalently, the gradient process $\{\delta h(t,\cdot)\}_{t\ge 0}$ is a Markov process with at least one stationary distribution, where $\delta h(t,x) := (h(t, x+e_i)-h(t,x), \, i=1,\ldots,d)$. Moreover, the stationary distribution of the gradient process can be chosen to be translation invariant on the lattice, as well as invariant under lattice symmetries, and with finite $L^1$ norm at each site. 
\end{thm}
Given Theorem \ref{mainthm}, a question that immediately comes to mind is whether the invariant measure is unique. We do not have an answer to this question, and leave it as an open problem. Nevertheless, there are some strong indications that the invariant measure is not unique. These are presented in Subsection \ref{unique} below.

Theorem \ref{mainthm} proves the existence of a stationary distribution, but does not tell us how to generate from a stationary distribution if we would like to do so on a computer. One problem is that computers cannot generate from an infinite lattice. To get around this issue, we formulate the next result. Instead of the whole lattice, consider the finite box $B_N := [-N,N]^d \cap \zz^d$. Starting with $f(0,x)=0$ for all $x\in\zz^d$, update $f$ at discrete times as follows. At each time $n$, draw a vertex $x$ uniformly at random from $B_N$. Then define
\[
f(n+1,x) := \max\biggl\{\max_{1\le i\le d} f(n, x\pm e_i), \, f(n,x)+1\biggr\}.
\]
Define $f(n+1,y)=f(n,y)$ for all other $y$. Note that although $f$ is defined on the whole lattice, it remains zero outside $B_N$ at all times. This is how ballistic deposition is usually simulated on a computer (e.g., in~\cite{pagnaniparisi15}). 
\begin{thm}\label{simulthm}
Let $f$ be defined as above. Given $p\in (0,1)$, generate $n$ from the geometric distribution with success probability $p$, and let $u_{p}(x) := f(n,x) - f(n,0)$. Suppose that $\{p(N)\}_{N\ge 1}$ is a sequence of numbers in $(0,1)$ such that $N^{-d-1}\ll p(N)\ll N^{-d}$ as $N\to \infty$. Then the sequence $\{u_{p(N)}\}_{N\ge 1}$ is a tight family in $\Omega$, and every weakly convergent subsequence converges to a stationary distribution of $\{h(t,\cdot)-h(t,0)\}_{t\ge 0}$, where $h$ is the continuous-time ballistic deposition process defined earlier.
\end{thm}
The above theorem allows us to generate approximately from a stationary distribution of $\{h(t,\cdot)-h(t,0)\}_{t\ge 0}$ by choosing some large $N$ and some $p$ such that $N^{-d-1}\ll p\ll N^{-d}$, and then generating the surface $u_{p}$ as above. The results from a simulation with $d=1$, $N =1000$ and $p=0.0001$ are shown in Figure \ref{simulfig}.

\begin{figure}[t]
\centering
\begin{subfigure}{.65\textwidth}
\includegraphics[width = \textwidth]{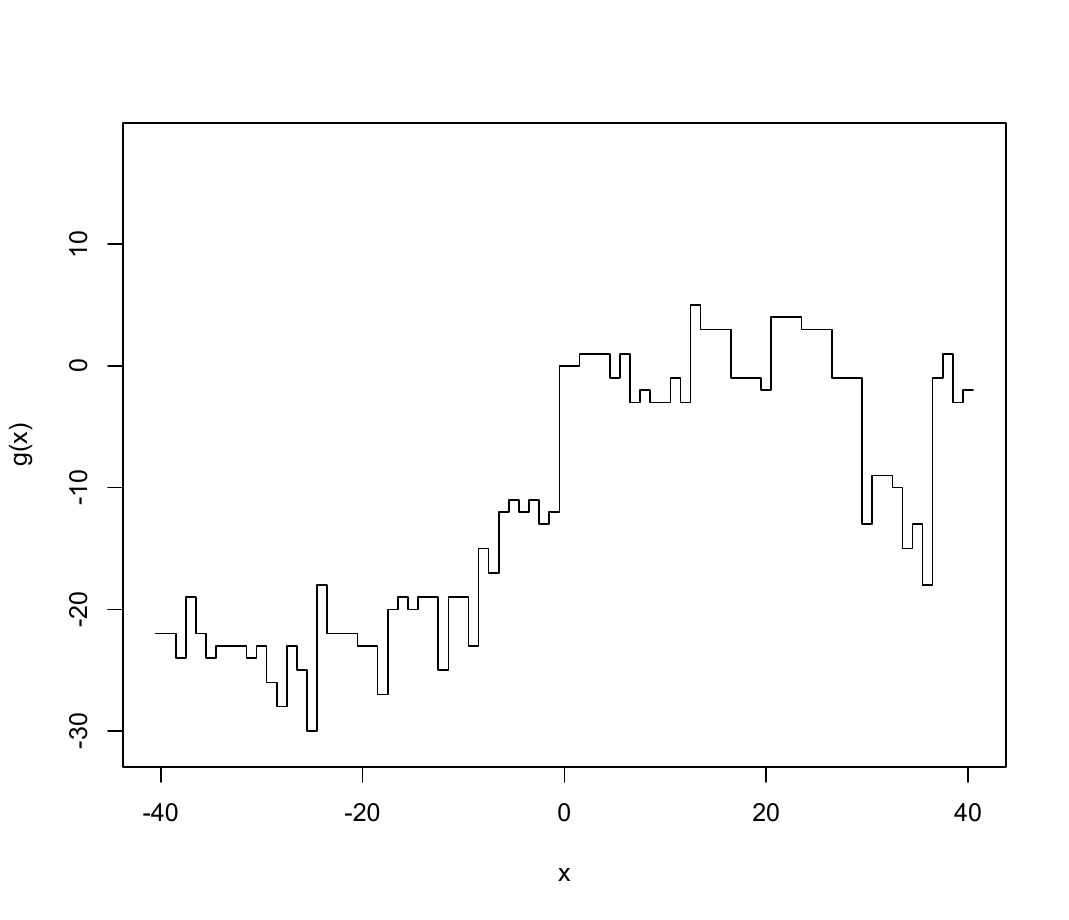}
\caption{Height function}
\end{subfigure}
\begin{subfigure}{.65\textwidth}
\includegraphics[width = \textwidth]{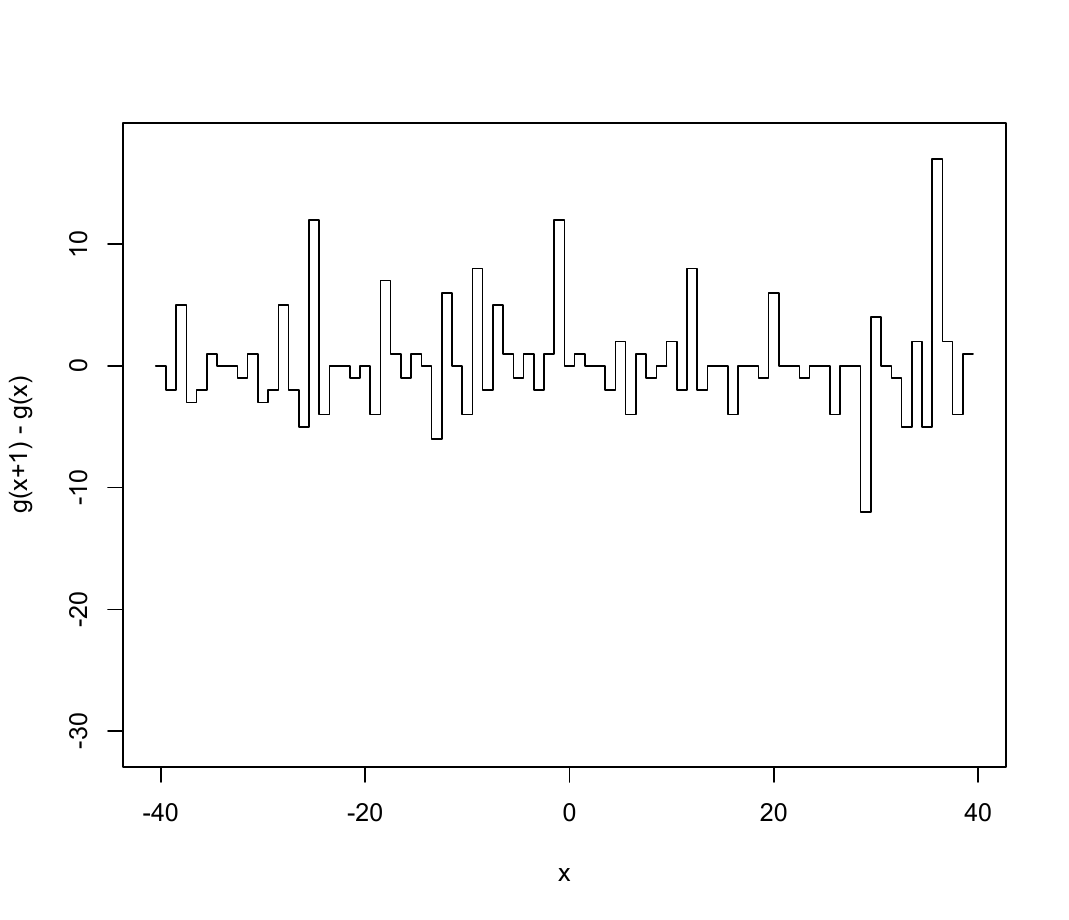}
\caption{Gradient}
\end{subfigure}
\caption{An approximately stationary ballistic deposition surface $g$, generated by the algorithm from Theorem \ref{simulthm}, with $d=1$, $N=1000$, and $p = 0.0001$. For visual clarity, the height function and the gradient are plotted in the range $-40\le x\le 40$.\label{simulfig}}
 \end{figure}

\begin{remark}
Note that the process considered in Theorem \ref{simulthm} is a discrete-time process, which scales up the continuous time of Theorem \ref{mainthm} by a factor of $N^d$. That is, a unit interval of time in the continuous process corresponds to roughly $N^d$ updates in the discrete process. In Theorem \ref{simulthm}, when $n$ is generated from a geometric distribution with mean $1/p$, it corresponds to the continuous-time process at time of order $1/(pN^d)$. So, the condition $p\ll N^{-d}$ corresponds to the large-time asymptotics of the continuous process. Additionally, we need the condition $p\gg N^{-d-1}$ to ensure that $1/(pN^d) \ll N$, which is needed because the effect of the boundary in the continuous-time process in the box $B_N$ is not felt at the center before time $N$. In other words, to make the discrete process on $B_N$ behave like the continuous process on $\zz^d$ far away from the boundary of $B_N$, we need to restrict ourselves to discrete times $\ll N^{d+1}$.
\end{remark}

\subsection{Review of the literature}\label{litsec}
The interest in ballistic deposition stems partly from the fact that it is viewed as a toy model of diffusion limited aggregation~\cite{ataretal01}. Another reason why models such as ballistic deposition have attracted a lot of current interest is that they are believed to be in the Kardar--Parisi--Zhang (KPZ) universality class~\cite{quastel12, corwin16, cannizzarohairer20}. While some of these models are integrable, most, like ballistic deposition, remain mathematically intractable. 

Rigorous mathematical results about ballistic deposition are so few that they can be summarized in one paragraph. A strong law of large numbers for the height function of a variant of ballistic deposition on a finite one-dimensional strip was proved by~\citet{ataretal01}. \citet{seppalainen00} proved the general strong law of large numbers for ballistic deposition on $\zz^d$ for all $d$. \citet{penroseyukich01, penroseyukich02} proved central limit theorems and other results for the number of blocks deposited in a large region within a fixed time. \citet{penrose08} showed that when $d=1$, the fluctuations of $h(t,x)$ are at least of order $\sqrt{\log t}$ as $t\to\infty$. Refinements and extensions of the results of \citet{ataretal01} for ballistic deposition on  finite graphs were recently proved by~\citet{mansouretal19} and \citet{braun20}. A variant of ballistic deposition, called $0$-ballistic deposition, where the height at a location $x$ is updated to a randomly selected element of $\{h(t, x\pm e_i), i=1,\ldots,d\}\cup \{h(t,x)+1\}$, was recently analyzed in great depth by \citet{cannizzarohairer20}. Lastly, in another recent work~\cite{chatterjee21b}, a bound of order $\sqrt{t/\log t}$ on the fluctuations of $h(t,x)$ was established for a variant of ballistic deposition where all heights are updated simultaneously at integer times, but the block sizes are random. This is more or less the totality of all that has been rigorously proved about ballistic deposition until now. 

It is natural to investigate the existence of a stationary probability measure for ballistic deposition (or, more accurately, for the gradient process). Indeed, ballistic deposition is an interacting particle system, and the major portion of the classical literature on interacting particle systems studies the existence and properties of invariant distributions~\cite{liggett85}. The question is challenging for interface growth models like ballistic deposition because the classical techniques are designed mainly for processes that take value in a product of compact spaces, which is not the case for interface growth. Non-compact systems are amenable to analysis when they are linear, as in \cite[Chapter IX]{liggett85}, but ballistic deposition is highly nonlinear. More advanced modern methods (see~\cite{komorowskietal10} and references therein) also do not seem to be applicable for ballistic deposition.

Besides being a natural question from the point of view of interacting particle systems, a second motivation for this work comes from the slew of recent papers on the existence of stationary evolutions of the KPZ and related equations~\cite{bakhtinli19, dunlapetal21, dunlap20, corwinknizel21, bryckuznetsov21, brycetal21, barraquandledoussal21, pimentel18, pimentel21a, pimentel21b, yang21}. Although these results are about stochastic PDEs, and Theorem \ref{mainthm} is about a discrete process, there are some similarities at a conceptual level. The fact that ballistic deposition is conjectured to be in the KPZ universality class~\cite{pagnaniparisi15} provides a natural connection with the SPDE literature. In dimension one, a possible rigorous formulation of this conjecture is that under the KPZ scaling, the height function of stationary 1D ballistic deposition converges to that of the KPZ fixed point characterized by~\citet{matetskietal21}. 

\subsection{The question of uniqueness}\label{unique}
We make the following conjecture, and then present some evidences in favor of it. 
\begin{conj}\label{mainconj}
The Markov process $\{w(t,\cdot)\}_{t\ge 0}$ from Theorem \ref{mainthm} has multiple invariant probability measures.
\end{conj}
Recall that a Markov semigroup $\{\cp_t\}_{t\ge 0}$ is said to have the {\it Feller property} if for any bounded continuous function $f$ from the state space into the real line, and any $t>0$, $\cp_t f$ is also a bounded continuous function. Under some mild conditions, the Feller property implies the existence of invariant measures. On the other hand, $\{\cp_t\}_{t\ge 0}$ is said to have the {\it strong Feller property} if for any $t>0$, $\cp_t$ sends bounded measurable functions to bounded continuous functions. The strong Feller property, under some additional assumptions, implies that there is a unique invariant measure.  We have the following result, which is our first piece of evidence in support of Conjecture \ref{mainconj}.
\begin{thm}\label{fellerthm}
The Markov process $\{w(t,\cdot)\}_{t\ge 0}$ from Theorem \ref{mainthm} has the Feller property, but is not strongly Feller.
\end{thm}
The next piece of evidence is the following result, which gives a believable condition, which, if true, would imply the existence of multiple invariant probability measures for the process $\{w(t,\cdot)\}_{t\ge 0}$.
\begin{thm}\label{uithm}
Consider process $\{w(t,\cdot)\}_{t\ge 0}$ from Theorem \ref{mainthm} with initial state $w(0,x) = x_1$, where $x_1$ is the first coordinate of the vector $x$. With this initial state, the process satisfies
\begin{align}\label{newavg}
\sup_{T\ge 1}\frac{1}{T}\int_0^T \ee|w(t,e_1)| dt< \infty.
\end{align}
If it is further true that 
\begin{align}\label{newavg2}
\lim_{R\to\infty} \sup_{T\ge 1}\frac{1}{T}\int_0^T \ee(|w(t,e_1)| 1_{\{|w(t,e_1)| \ge R\}})dt = 0,
\end{align}
then there exist more than one invariant probability measure for the Markov process $\{w(t,\cdot)\}_{t\ge 0}$.
\end{thm}
Note that another way to state the unproven condition \eqref{newavg2} is that if $U(T)$ is drawn uniformly at random from the interval $[0,T]$, then the collection of random variables $\{w(U(T), e_1)\}_{T\ge 1}$ is uniformly integrable. The proven condition \eqref{newavg} says that this collection has uniformly bounded $L^1$ norm.

\subsection{Other open questions}\label{opensec}
Besides the question of uniqueness discussed in the previous subsection, Theorem \ref{mainthm} suggests many other follow-up questions, none of which seem easy to answer. The following is a partial list.
\begin{enumerate}
\item Can something be said about the rate of convergence to stationarity? 
\item Can at least one invariant measure be explicitly described, even for $d=1$? 
\item If $g$ is a random interface drawn from an invariant measure, what can be said about the rate of growth of $|g(x)|$ as $|x|\to \infty$? 
\item If $g$ is drawn from an invariant measure, does $g$ have a nontrivial scaling limit as the lattice spacing is sent to zero and the heights are scaled suitably? In particular, in dimension one, is it true that under the KPZ scaling, the process converges to the KPZ fixed point characterized by~\citet{matetskietal21}? 
\end{enumerate}
This concludes the discussion in this introductory section, and the statements of results and open questions. The rest of the paper is devoted to proofs. 


\section{Proof of Theorem \ref{mainthm}}\label{proofsec}
The proof of Theorem \ref{mainthm} is inspired by the methods of \cite{chatterjee21c} (see also the related works~\cite{chatterjee21, chatterjee21b, chatterjee21b1, chatterjeesouganidis21}). A key ingredient in the proof is the following simple lemma about real numbers. 
\begin{lmm}\label{mainlmm}
For any $n\ge 2$ and real numbers $x_1,\ldots,x_n$, 
\[
\max_{1\le i\le n} x_i - \frac{1}{n}\sum_{i=1}^n x_i \ge \frac{1}{2n(n-1)} \sum_{1\le i< j\le n} |x_i - x_j|. 
\]
\end{lmm}
\begin{proof}
Let $x^+ := \max\{x,0\}$ and $x^-:= -\min\{x,0\}$ denote the positive  and negative parts of a real number $x$. Note that $x=x^+ - x^-$ and $|x|=x^+ + x^-$. For each $i$, let 
\[
y_i := x_i - \frac{1}{n}\sum_{j=1}^n x_j,
\]
so that
\begin{align}\label{step1}
\max_{1\le i\le n} x_i - \frac{1}{n}\sum_{i=1}^n x_i = \max_{1\le i\le n } y_i. 
\end{align}
Since $\sum_{i=1}^n y_i = 0$, at least one $y_i$ must be nonnegative. Thus,
\begin{align*}
 \max_{1\le i\le n } y_i &=  \max_{1\le i\le n } y_i^+ \ge \frac{1}{n}\sum_{i=1}^n y_i^+. 
\end{align*}
But again, since $\sum_{i=1}^n y_i = 0$,
\[
\sum_{i=1}^n y_i^+ = \sum_{i=1}^n y_i^-.
\]
Combining the last two displays, we get
\begin{align}\label{step2}
 \max_{1\le i\le n } y_i  &\ge \frac{1}{2n}\sum_{i=1}^n (y_i^+ + y_i^-) = \frac{1}{2n}\sum_{i=1}^n |y_i|. 
\end{align}
Now, for any $i$ and $j$,
\begin{align*}
|x_i - x_j| &= |y_i - y_j| \le |y_i| + |y_j|. 
\end{align*}
Thus,
\begin{align*}
\sum_{1\le i\ne j\le n} |x_i - x_j| &\le \sum_{1\le i\ne j\le n} (|y_i|+|y_j|) = 2(n-1) \sum_{i=1}^n |y_i|. 
\end{align*}
Combining this with \eqref{step1} and \eqref{step2} completes the proof. 
\end{proof}
Let $h(t,x)$ denote the height of the ballistic deposition surface at time $t$ and location $x$, starting with $h(0,x)=0$ for all $x$. Clearly, the law of $h(t,\cdot)$ is translation-invariant. In particular,
\[
\alpha(t) := \ee(h(t,x))
\]
does not depend on $x$. 
The following lemma follows from \cite[Proposition 2.1]{penrose08}.
\begin{lmm}[See Proposition 2.1 in \cite{penrose08}]\label{penroselmm}
There is a constant $C(d)$ depending only on $d$, such that for all $t$, $\alpha(t)\le C(d)t$. 
\end{lmm}
Define another function
\[
\beta(t) := \ee\biggl[\max\biggl\{\max_{1\le i\le d} h(t,\pm e_i), \, h(t,0)\biggr\}\biggr]. 
\]
The following lemma is a consequence of the definition of ballistic deposition. 
\begin{lmm}\label{derivlmm}
For any $t,\delta\ge 0$, 
\begin{align*}
\alpha(t+\delta)-\alpha(t) \ge\delta e^{-(2d+1)\delta}(\beta(t)-\alpha(t)).
\end{align*}
\end{lmm}
\begin{proof}
Suppose that in the time interval $[t,t+\delta)$, there is exactly one deposition at $0$, and none at any of the neighbors of $0$. Let $E$ denote this event. If this event takes place, then
\[
h(t+\delta,0) = \max\biggl\{\max_{1\le i\le d} h(t,\pm e_i), \, h(t,0)+1\biggr\}. 
\]
Thus, if $\mf_t$ denotes the $\sigma$-algebra generated by all updates in the time interval $[0,t)$, then by the monotonicity of $h(t,0)$, 
\begin{align*}
&\ee(h(t+\delta,0)-h(t,0)) \\
&\ge \ee((h(t+\delta,0)-h(t,0))1_E\mid \mf_t)\\
&= \ee\biggl(\max\biggl\{\max_{1\le i\le d} (h(t,\pm e_i)- h(t,0)), \, 1\biggr\}1_E\, \biggl|\,  \mf_t\biggr)\\
&= \max\biggl\{\max_{1\le i\le d} (h(t,\pm e_i)- h(t,0)), \, 1\biggr\}\pp(E\mid\mf_t),
\end{align*}
where the last step holds because the term that was taken out of the conditional expectation is $\mf_t$-measurable and  integrable  (by Lemma \ref{penroselmm}). Noting that 
\[
\pp(E\mid \mf_t) = \delta e^{-(2d+1)\delta},
\]
and taking expectation on  both sides in the previous display, we get
\begin{align*}
&\ee(h(t+\delta,0)-h(t,0)) \\
&\ge \delta e^{-(2d+1)\delta}\ee\biggl(\max\biggl\{\max_{1\le i\le d} (h(t,\pm e_i)- h(t,0)), \, 1\biggr\}\biggr)\\
&= \delta e^{-(2d+1)\delta}\biggl[\ee\biggl(\max\biggl\{\max_{1\le i\le d}h(t,\pm e_i), \, h(t,0) + 1\biggr\} - h(t,0)\biggr)\biggr]\\
&\ge \delta e^{-(2d+1)\delta}(\beta(t)-\alpha(t)).
\end{align*}
This completes the proof of the lemma.
\end{proof}
We need the following standard result about Riemann integrability of bounded monotone functions. The proof is omitted.
\begin{lmm}\label{intlmm}
Take any $t\ge 0$, and a bounded monotone function $f:[0,t]\to \rr$. Then 
\begin{align*}
\lim_{n\to\infty} \frac{t}{n}\sum_{i=0}^{n-1} f(it/n) = \int_0^t f(s)ds.
\end{align*}
\end{lmm}
Combining Lemma \ref{derivlmm} and Lemma \ref{intlmm} yields the following result. This result is conceptually similar to  \cite[Proposition 3.1]{dunlap20}, \cite[Proposition 5.2]{dunlapetal21} and \cite[Lemma 4.3]{chatterjee21c}, which were used to construct  stationary solutions of the stochastic Burgers equation.
\begin{lmm}\label{integratelmm}
For any $t\ge 0$,
\begin{align*}
\alpha(t) \ge \int_0^t (\beta(s)-\alpha(s))ds. 
\end{align*}
\end{lmm}
\begin{proof}
Fix $t$ and take any $n$. Then by Lemma \ref{derivlmm} and the fact that $\alpha(0)=0$, 
\begin{align*}
\alpha(t) &= \alpha(t)-\alpha(0)\\
&= \sum_{i=0}^{n-1} (\alpha((i+1)t/n)  - \alpha(it/n))\\
&\ge \frac{t}{n}e^{-(2d+1)t/n} \sum_{i=0}^{n-1} (\beta(it/n) -\alpha(it/n)). 
\end{align*}
Since $\alpha$ and $\beta$ are monotone functions, and are bounded on any finite interval (by Lemma \ref{penroselmm}), the proof is now completed by sending $n\to \infty$ and applying Lemma~\ref{intlmm}.
\end{proof}
Combining Lemma \ref{mainlmm} and Lemma \ref{integratelmm}, we arrive at the following key result. 
\begin{lmm}\label{bdlmm}
For all $t\ge0$, and any $i$,
\[
\int_0^t \ee|h(s,e_i)-h(s,0)| ds \le C(d) t, 
\]
where $C(d)$ depends only on $d$.
\end{lmm}
\begin{proof}
By translation invariance of the law of $h(s,\cdot)$,
\begin{align*}
\alpha(s) = \frac{1}{2d+1}\ee\biggl(h(s,0)+\sum_{i=1}^dh(s, \pm e_i)\biggr). 
\end{align*}
On the other hand, by Lemma \ref{mainlmm},
\begin{align*}
&\max\biggl\{\max_{1\le i\le d} h(s,\pm e_i), \, h(s,0)\biggr\} - \frac{1}{2d+1}\biggl(h(s,0)+\sum_{i=1}^dh(s, \pm e_i)\biggr)\\
&\ge \frac{1}{8d(2d+1)} \sum_{a,b\in A} |h(s,a)-h(s,b)|,
\end{align*}
where $A:= \{0,\pm e_1,\ldots, \pm e_d\}$. But 
\[
\sum_{a,b\in A} |h(s,a)-h(s,b)| \ge \sum_{i=1}^d |h(s, \pm e_i)-h(s,0)|. 
\]
Combining the last three displays, and using the definitions of $\alpha$ and $\beta$, we get that for any $j$,
\begin{align*}
\beta(s)-\alpha(s) &\ge \frac{1}{8d(2d+1)}\sum_{i=1}^d \ee|h(s, \pm e_i)-h(s,0)| \\
&= \frac{\ee|h(s,e_j)-h(s,0)|}{4(2d+1)}.
\end{align*}
Integrating both sides from $0$ to $t$, and applying Lemma \ref{integratelmm} and Lemma \ref{penroselmm} completes the proof. 
\end{proof}
We will use Lemma \ref{bdlmm} to establish the tightness of a collection of $\Omega$-valued random variables. For that purpose, we need the following simple criterion.
\begin{lmm}\label{tightlmm}
Let $I$ be an arbitrary set. A collection of $\Omega$-valued random variables $\{f_i\}_{i\in I}$  is tight if and only if $\{f_i(x)\}_{i\in I}$ is a tight family of real-valued random variables for every $x\in \zz^d$.
\end{lmm}
\begin{proof}
If $\{f_i\}_{i\in I}$ is a tight family, then the continuity of the projection $f\mapsto f(x)$ shows that for any $x$, $\{f_i(x)\}_{i\in I}$ is a tight family. Conversely, suppose that $\{f_i(x)\}_{i\in I}$ is a tight family for each $x$. Fix some $\delta>0$. Then for every $x\in \zz^d$, there exists $C_x>0$ such that $\pp(|f_i(x)| > C_x)\le 2^{-|x|}\delta$ for all $i\in I$. Let $K := \prod_{x\in \zz^d} ([-C_x, C_x]\cap \zz)$. Then $K$ is a compact subset of $\Omega$ under the product topology, and for any $i\in I$, 
\[
\pp(f_i \notin K) \le \sum_{x\in \zz^d} \pp(|f_i(x)|> C_x) \le C\delta,
\]
where $C$ does not depend on $n$. This completes the proof.
\end{proof}

For each $t$, let $\mu_t$ be the law of $h(s,\cdot)- h(s,0)$ where $s$ is chosen uniformly at random from $[0,t]$, and $h$ is our  ballistic deposition process with zero initial condition. Using Lemma \ref{bdlmm} and Lemma \ref{tightlmm}, we deduce the following result.
\begin{lmm}\label{maintight}
$\{\mu_t\}_{t\ge 0}$ is a tight family of probability measures on $\Omega$. 
\end{lmm}
\begin{proof}
Let $g_t$ be an $\Omega$-valued random variable with law $\mu_t$. By Lemma \ref{bdlmm}, we have that for any $i$,
\begin{align}\label{gtbound}
\ee|g_t(e_i)| = \frac{1}{t}\int_0^t \ee|h(s,e_i)-h(s,0)| ds \le C(d). 
\end{align}
But by the translation invariance of the ballistic deposition process (since we are starting from an all zero initial condition), it is easy to see that for any $x$, 
\[
\ee|g_t(e_i)| = \ee|g_t(x+e_i)-g_t(x)|.
\]
For any $x\in\zz^d$, $g_t(x)$ can be written as the sum of the successive differences in the value of $g_t$ along a path connecting $0$ to $x$. Combining, we get that $\ee|g_t(x)|$ is bounded above by a number that may depend on $x$ and $d$, but not on $t$. Thus, $\{g_t(x)\}_{t\ge 0}$ is a tight family of random variables. By Lemma \ref{tightlmm}, this shows that $\{\mu_t\}_{t\ge 0}$ is a tight family of probability measures on $\Omega$. 
\end{proof}
Since $\Omega$ is a Polish space under the product topology, Lemma \ref{maintight} shows that $\{\mu_t\}_{t\ge 0}$ has subsequential weak limits on the space of probability measures on $\Omega$. In the remaining part of the proof of Theorem \ref{mainthm}, we will show that any subsequential weak limit is a stationary probability measure for the Markov process $\{h(t,\cdot) - h(t,0)\}_{t\ge 0}$. Note that we have not yet shown that the process is Markovian. We do this below.  

Fix some $T>0$. The updates up to time $T$ are governed by independent Poisson point processes on $[0,T]$, one for each site $x$. Jointly, these processes take value on the state space
\[
\Gamma := \prod_{x\in \zz^d} \Gamma_x,
\]
where, for each $x$, $\Gamma_x$ is the set of all finite subsets of $[0,T]$. Each $\Gamma_x$ is the union of $\Gamma_{x,n}$ over $n\ge 0$, where $\Gamma_{x,n}$ is the set of all subsets of $[0,T]$ of size $n$. Let us view $\Gamma_{x,n}$ as the collection of ordered $n$-tuples of distinct elements of $[0,T]$. This makes $\Gamma_{x,n}$ a subset of $[0,T]^n$, which induces a metric topology on $\Gamma_{x,n}$. Let $\Gamma_x$ be given  the topology generated by the open subsets of $\Gamma_{x,n}$ over all $n$. Finally, let $\Gamma$ be assigned the product topology.

Let $\lambda$ denote the joint law of our Poisson processes on $\Gamma$. Explicitly, $\lambda$ is the product of $\lambda_x$ over all $x$, where each $\lambda_x$ is the law of the Poisson process at site $x$, defined as follows. Each $\Gamma_{x,n}$ is given the Lebesgue measure inherited from $[0,T]^n$, normalized to have total mass $e^{-T} T^n/n!$. Since $\Gamma_x$ is the disjoint union of $\Gamma_{x,n}$ over $n\ge 0$, and the sum of $e^{-T} T^n/n!$  over $n\ge0$ equals $1$, this defines a probability measure on $\Gamma_x$, which is our $\lambda_x$.

A typical element of $\Gamma$ will be denoted by $P = (P_x)_{x\in \zz^d}$. Let $\Gamma'$ be the subset of $\Gamma$ consisting of all $P$ such that $P_x$ and $P_y$ are disjoint for every pair of distinct $x$ and $y$. Then it is easy to see that $\Gamma'$ is a measurable subset of $\Gamma$, and $\lambda(\Gamma')=1$. 

Given $f\in \Omega$ and $P\in \Gamma'$, let $\Psi(f,P)$ be the result of applying the ballistic deposition updates prescribed by $P$ to the initial condition $f$. Explicitly, this is defined as follows. Take any finite $D\subseteq \zz^d$. Since the union of $P_x$ over $x\in D$ is a finite subset of $[0,T]$, its elements can arranged in increasing order. Then, update $f$ inductively at these times according to ballistic deposition growth, but perform no updates outside $D$. Call the resulting surface $\Psi_D(f,P)$. It is easy to see that $\Psi_D(f,P)$ is monotone increasing in $D$. Thus, the pointwise limit of $\Psi_D(f,P)$ as $D$ increases to $\zz^d$ exists. Let us denote this limit by $\Psi(f,P)$. Note that this limit is in $(\zz\cup \{\infty\})^{\zz^d}$, but not necessarily in $\Omega$.

Given some $P\in \Gamma'$, we will now define for each $x\in \zz^d$ a set $S(x,P)\subseteq \zz^d$. To define $S(x,P)$, we first define inductively a sequence of numbers $T = t_0(x,P)> t_1(x,P) > \cdots \ge 0$ and a sequence of finite sets $\{x\}=S_0(x,P)\subseteq S_1(x,P)\subseteq \cdots$, as follows. Given $S_k(x,P)$ and $t_k(x,P)$, let $t$ be the largest number in $[0, t_k(x,P))$ such that $t\in P_y$ for some $y\in S_k(x,P)$. In case there exists such a $t$ and $y$, define $t_{k+1}(x,P) := t$, and $S_{k+1}(x,P)$ to be the union of $S_k(x,P)$ and the set of neighbors of $y$. In case there does not exist such $t$ and $y$ (which includes the case $t_k(x,P)=0$), stop, and define $S(x,P):= S_k(x,P)$. Also, to keep track of the stopping time, define $K(x,P):= k$. If this process never stops, let $S(x,P)$ be the union of $S_k(x,P)$ over all $k$, and let $K(x,P):= \infty$. 

Intuitively, $S(x,P)$ is the set such that the height at $x$ at time $t$ is determined by the heights at $y\in S(x,P)$ at time $0$ and the updates given by $P$. This will be made precise in the proof of Lemma \ref{psiimage} below, and also in Lemma \ref{slmm}.


Let $\Gamma''$ be the set of all $P\in \Gamma'$ such that $K(x,P)< \infty$ for all $x\in \zz^d$. We need the following facts about $\Gamma''$.
\begin{lmm}\label{gmeas}
$\Gamma''$ is a measurable subset of $\Gamma$.
\end{lmm}
\begin{proof}
It suffices to show that the set $\{P:K(x,P)<\infty\}$ is measurable for any given $x$. Fix some $x$. Then this set is the union over all finite $D\subseteq \zz^d$ and $k\ge 0$ of the events $\{P: S(x,P)=D, \, K(x,P)=k\}$. For a given $D$ and $k$, this event can be explicitly written down in terms of $\{P_y\}_{y\in D}$. Since $D$ is finite, it is now easy to check that this is a measurable set. 
\end{proof}

\begin{lmm}\label{psiimage}
$\Psi$ maps $\Omega\times \Gamma''$ into $\Omega$. 
\end{lmm}
\begin{proof}
Take any $x\in \zz^d$, $f\in \Omega$ and $P\in \Gamma''$. For simplicity, let us write $t_k$, $S_k$ and $K$ instead of $t_k(x,P)$, $S_k(x,P)$ and $K(x,P)$. Since there are no updates at any element of $S_K$ in the time interval $[0,t_K)$, the heights at $S_K$ undergo no change in this interval. When the first update happens at time $t_K$, it only affects heights at $S_{K-1}$. Moreover, since the neighbors of the element of $S_{K-1}$ where the update occurs at time $t_K$ are elements of $S_K$, it follows that the maximum height in $S_{K-1}$ after the update is at most the maximum height in $S_K$ before the update plus one. By a similar argument, we have that for each $k$, the maximum height in $S_{k-1}$ after time $t_k$ is at most the maximum height in $S_k$ before time $t_k$ plus one. This proves that the final height at $x$ is finite.
\end{proof}

\begin{lmm}\label{contlmm}
$\Psi$ is a continuous map on $\Omega \times \Gamma''$. 
\end{lmm}
\begin{proof}
Take any $x\in \zz^d$ and $j\in \zz^d$, and let $\Omega_{x,j} := \{f\in \Omega: f(x)=j\}$. Since these sets form a subbase for the topology on $\Omega$, it suffices to show that $\Psi^{-1}(\Omega_{x,j})$ is relatively open in $\Omega \times \Gamma''$. Accordingly, take any $(f,P)$ in this inverse image. Let $g := \Psi(f,P)$, so that $g(x)=j$. Let $S := S(x,P)$, which is finite since $P\in \Gamma''$. Take any $f'\in \Omega$  such that $f'(y)=f(y)$ for all $y\in S(x,P)$, and any $P'\in \Gamma''$ such that the relative temporal order and spatial locations of the updates in $\cup_{y\in S} P_y'$ are the same as that in $\cup_{y\in S} P_y$. Let $g':= \Psi(f',P')$. Then by the same inductive argument as in the proof of Lemma \ref{psiimage}, we get that $g'(x)=j$. Moreover, the set of all $f'$ with the above property is an open subset of $\Omega$, and the set of all $P'\in \Gamma''$ with the above property is a relatively open subset of $\Gamma''$. This proves that $\Psi^{-1}(\Omega_{x,j})$ is a relatively open subset of $\Omega \times \Gamma''$. 
\end{proof}

\begin{lmm}\label{gprob}
$\lambda(\Gamma'') = 1$. 
\end{lmm}
\begin{proof}
Let $P$ be a random element of $\Gamma$ with law $\lambda$. Then each $P_x$ is a Poisson process on $[0,T]$, and the $P_x$'s are independent. Fix some $x$, and for simplicity, let us write $t_k$, $S_k$ and $K$ instead of $t_k(x,P)$, $S_k(x,P)$ and $K(x,P)$. Let us extend the definition of $t_k$ as follows. First, suppose that we have independent Poisson processes $P_x'$ on $(-\infty, T]$ instead of $[0,T]$. Then, given $S_k$, define $t_{k+1}$ to be the largest $t\in (-\infty, t_k)$ such that $t\in P_y'$ for some $y\in S_k$. Then the new $t_k$ is the same as the old one up to $k=K$, and $K+1$ is the smallest $j$ such that $t_j < 0$. 

For each $t\in (-\infty, T]$, let $\mf_t$ be the $\sigma$-algebra generated by $(P_y'\cap[t, T])_{y\in \zz^d}$. Then $\{\mf_t\}_{t\in (-\infty, T]}$ is a filtration of $\sigma$-algebras, with time moving in the reverse direction (that is, $\mf_s \subseteq \mf_t$ if $s\ge t$). Note that each $t_k$ is a stopping time with respect to this filtration. Let $\mf_{t_k}$ be the stopped $\sigma$-algebra of $t_k$, meaning that an event $A$ is in $\mf_{t_k}$ if and only if for any $t\in (-\infty, T]$, the event $A\cap \{t_k \ge t\}$ is in $\mf_t$. Then, by the memoryless property of Poisson processes, the conditional distribution of $t_k-t_{k+1}$ given $\mf_{t_k}$ is exponential with mean $1/|S_k|$, because it is the minimum of $|S_k|$ independent exponential random variables with mean $1$. This lets us compute
\begin{align*}
\ee(e^{-(t_k-t_{k+1})}\mid \mf_{t_k}) &= \frac{|S_k|}{|S_k|+1}. 
\end{align*}
By construction, $|S_k|\le 2dk + 1$. Thus, by the above display and the fact that $x\mapsto x/(x+1)$ is an increasing map on $(0,\infty)$, we get
\[
\ee(e^{-(t_k-t_{k+1})}\mid \mf_k) \le \frac{2dk+1}{2dk+2}. 
\]
From this, we easily get
\begin{align*}
\ee(e^{t_k - T}) &= \ee\biggl(\prod_{j=0}^{k-1} e^{t_{j+1}-t_j}\biggr) \le \prod_{j=0}^{k-1} \frac{2dj + 1}{2dj+2}. 
\end{align*}
Since $\{t_k\}_{k\ge 0}$ is a decreasing sequence in $(-\infty, T]$, the limit $t_\infty:= \lim_{k\to\infty} t_k$ exists in $[-\infty, t]$. Thus, by the above inequality and the dominated convergence theorem, we get
\[
\ee(e^{t_\infty}) \le e^T \prod_{j=0}^\infty \frac{2dj + 1}{2dj+2} = 0.
\]
Thus, $t_\infty = -\infty$ almost surely, which means that $K<\infty$ almost surely. Consequently, $S(x,P)$ is finite almost surely. Since this holds for every $x$ and $\zz^d$ is countable, we get that $\lambda(\Gamma'')=1$.
\end{proof}

To summarize, we have produced a measurable set $\Gamma''\subseteq \Gamma$ and a continuous map $\Psi :\Omega \times \Gamma''\to\Omega$, such that for any $f\in \Omega$ and $P\in \Gamma''$, $\Psi(f, P)$ is the ballistic deposition surface at time $T$ with initial condition $f$ and update times given by $P$. Moreover, we have shown that  the law $\lambda$ of our Poissonian  update times is supported on $\Gamma''$.  As a first application, we show the following.   
\begin{lmm}\label{markovlmm}
Let $w(t,x) := h(t,x)-h(t,0)$. Then $\{w(t,\cdot)\}_{t\ge 0}$ is a time-homogeneous Markov process on $\Omega$.
\end{lmm}
\begin{proof}
Fix some $T> 0$, as above, and take any initial condition $h(0,\cdot)$. By the above construction, $h(T,\cdot)$ can be written as $\Psi(h(0,\cdot), P)$, where $P$ has law $\lambda$. It is not difficult to show that if the initial condition $h(0,\cdot)$ is replaced by $h(0,\cdot)-c$ for some constant $c$, then the height function at time $T$ becomes $h(T, \cdot)-c$. Taking $c = h(0,0)$, we see that with initial condition $w(0,\cdot)$, the height function at time $T$ is $h_1(T,\cdot) := h(T,\cdot)-h(0,0)$. If we subtract off the constant $h_1(T,0)$ from this height function, we get $w(T,\cdot)$, since 
\[
w(T,x) = h(T,x)-h(T,0) = h_1(T,x) - h_1(T,0). 
\]
In other words, $h_1(T,\cdot) = \Psi(w(0,\cdot), P)$ and $w(T,\cdot) =  h_1(T,\cdot) - h_1(T,0)$. Thus, $w(T,\cdot)$ can be expressed as a measurable function of $w(0,\cdot)$ and $P$. It is not hard to show using this that $\{w(t, \cdot)\}_{t\ge 0}$ is a time-homogeneous Markov process. 
\end{proof}
We are now ready to finish the proof of Theorem \ref{mainthm}. Let $\{\cp_t\}_{t\ge 0}$ denote the transition semigroup for $\{w(t,\cdot)\}_{t\ge 0}$. That is, for a function $\Phi:\Omega \to \rr$, $\cp_t \Phi(f)$ is the expected value of $\Phi(w(t,\cdot))$ given $w(0,\cdot)=f$, provided that this expectation is well-defined. 
Recall the tight family of probability measures $\{\mu_t\}_{t\ge 0}$ from Lemma \ref{maintight}. By Theorem \ref{fellerthm} (which will be proved later in Section \ref{fellerproofsec}), the semigroup $\{\cp_t\}_{t\ge 0}$ has the Feller property. Combining this result with the tightness of $\{\mu_t\}_{t\ge 0}$, and applying the Kyrlov--Bogoliubov theorem (e.g., \cite[Corollary 3.1.2]{dapratozabczyk96}), we get the existence of a stationary probability measure for the process $\{w(t,\cdot)\}_{t\ge 0}$.

For the process $\{\delta h(t,\cdot)\}_{t\ge 0}$, simply note that the two processes are in a simple one-to-one correspondence, because $w(t,\cdot)$ can be computed using $\delta h(t,\cdot)$, and $\delta h(t,\cdot)$ can be read off from $w(t,\cdot)$. Using this, it is easy to show that (a) $\{\delta h(t,\cdot)\}_{t\ge 0}$ is a time-homogeneous Markov process, and (b) if $g$ is drawn from a stationary distribution of $\{w(t,\cdot)\}_{t\ge 0}$, then the law of $\delta g$ is a stationary probability measure for $\{\delta h(t,\cdot)\}_{t\ge 0}$. The translation invariance and invariance under lattice symmetries follow from the construction of the stationary measure. Lastly, the integrability of $g(x)$ follows from inequality \eqref{gtbound} and the fact that if $\{X_n\}_{n\ge 1}$ is a sequence of  nonnegative real-valued random variables converging in law to $X$, then $\ee(X)\le \liminf_{n\to\infty} \ee(X_n)$. 

\section{Proof of Theorem \ref{simulthm}}\label{simulproof}
We will continue to use the notations introduced in Section \ref{proofsec}. In particular, $h$ denotes a ballistic deposition process on $\zz^d$ started from an all zero initial condition, $w(t,x) = h(t,x)-h(t,0)$, and $\{\cp_t\}_{t\ge0}$ is the transition semigroup for the Markov process $\{w(t,\cdot)\}_{t\ge 0}$. The dual operator $\cp_t^*$  acts on the space of probability measures on $\Omega$ in the usual way --- if $\nu$ is a probability measure on $\Omega$, then $\cp_t^* \nu$ is the law of $w(t,\cdot)$ if $w(0,\cdot)$ is drawn from $\nu$.  Given $a > 0$, let $\gamma_a$ denote the law of $w(s,\cdot)$, where $s$ is drawn from the exponential distribution with mean $a$. The proof of the following lemma is essentially a reworking of the proof of the Krylov--Bogoliubov theorem.
\begin{lmm}\label{tightnew}
$\{\gamma_a\}_{a>0}$ is a tight family of probability measures on $\Omega$, and any subsequential weak limit of this family, as $a\to \infty$, is a stationary measure for the semigroup $\{\cp_t\}_{t\ge 0}$.
\end{lmm}
\begin{proof}
Let $v_a$ be an $\Omega$-valued random variable with law $\gamma_a$. Using integration by parts and Lemma \ref{bdlmm}, we have that for any $i$,
\begin{align*}
\ee|v_a(e_i)| &= \int_0^\infty a^{-1}e^{-t/a} \ee|h(t,e_i)-h(t,0)| dt \\
&= \int_0^\infty a^{-2}e^{-t/a} \biggl(\int_0^t\ee|h(s,e_i)-h(s,0)| ds\biggr)dt \\
&\le C(d) \int_0^\infty a^{-2}t e^{-t/a} dt = C(d),
\end{align*}
where $C(d)$ depends only on $d$. 
Armed with this, and proceeding as in the proof of Lemma \ref{maintight}, we can now easily complete the proof of tightness. 

Next, recall that $\nu_s$ denotes the law of $w(s,\cdot)$. Fix some $T>0$ and  a bounded continuous function $\Phi:\Omega \to \rr$. Then for any $a>0$,
\begin{align*}
&\int_\Omega \Phi(f) d(\cp_T^*\gamma_a)(f)  - \int_\Omega \Phi(f) d\gamma_a(f) \\
&=\frac{1}{a} \int_0^\infty \int_\Omega \Phi(f) e^{-s/a} d(\cp_T^*\nu_s) (f)ds - \frac{1}{a}\int_0^\infty\int_\Omega \Phi(f) e^{-s/a} d\nu_s (f)ds\\
&=\frac{1}{a} \int_0^\infty\int_\Omega \Phi(f)e^{-s/a}  d\nu_{s+T} (f)ds -\frac{1}{a} \int_0^\infty\int_\Omega \Phi(f) e^{-s/a} d\nu_s (f) ds\\ 
&=\frac{1}{a}\int_T^\infty\int_\Omega \Phi(f) (e^{-(s-T)/a} - e^{-s/a}) d\nu_{s} (f)ds \\
&\qquad \qquad \qquad - \frac{1}{a}\int_0^T\int_\Omega \Phi(f) e^{-s/a} d\nu_s (f)ds. 
\end{align*}
Using the fact that $\Phi$ is bounded, it is now easy to show that both integrals in the last expression tend to zero as $a\to \infty$. Thus, 
\[
\lim_{a\to \infty} \biggl(\int_\Omega \Phi(f) d(\cp_T^*\gamma_a)(f)  - \int_\Omega \Phi(f) d\gamma_a(f)\biggr) = 0.
\]
By the tightness of $\{\gamma_a\}_{a>0}$, there is a sequence $a_k\to\infty$ and a probability measure $\gamma$ such that $\gamma_{a_k}\to\gamma$ weakly. Thus, 
\[
\lim_{k\to \infty} \int_\Omega \Phi(f) d\gamma_{a_k}(f) = \int_\Omega \Phi(f) d\gamma(f).
\]
On the other hand, by the Feller property of $\{\cp_t\}_{t\ge 0}$ (Theorem \ref{fellerthm}), 
\begin{align*}
\lim_{k\to \infty} \int_\Omega \Phi(f) d(\cp_T^*\gamma_{a_k})(f)  &= \lim_{k\to \infty} \int_\Omega P_T \Phi(f) d\gamma_{a_k}(f) \\
&= \int_\Omega P_T \Phi(f) d\gamma (f) = \int_\Omega \Phi(f) d(\cp_T^*\gamma)(f). 
\end{align*}
Thus, $\cp_T^* \gamma=\gamma$, which means that $\gamma$ is an invariant probability measure for the semigroup $\{\cp_t\}_{t\ge 0}$. 
\end{proof}
For the next lemma, we need the following crude lower tail bound for gamma random variables.
\begin{lmm}\label{gammalmm}
Let $X$ be a gamma random variable with shape parameter $n$ and scale parameter $1$. Then for any $a\in(0,1)$, $\pp(X\le an) \le e^{(1-a+\log a)n}$. 
\end{lmm}
\begin{proof}
Recall that $X$ can be written as a sum of $n$ i.i.d.~exponential random variables $X_1,\ldots,X_n$ with mean $1$. Thus, for any $a>0$ and $\theta >0$,
\begin{align*}
\pp(X\le an) &= \pp(e^{-\theta X} \ge e^{-\theta an}) \le e^{\theta an} \ee(e^{-\theta X})\\
&= e^{\theta an} (\ee(e^{-\theta X_1}))^n = e^{\theta an} (1+\theta)^{-n}. 
\end{align*}
The result is now obtained by choosing $\theta = (1-a)/a$, which optimizes the above bound. 
\end{proof}
Fix some $T>0$. Recall the definition of $S(x,P)$ and associated notations from Section \ref{proofsec}. Let $\rho(x,P)$ be the maximum $\ell^1$ distance of any vertex in $S(x,P)$ from~$x$. The following lemma is essentially a restatement of \cite[Proposition 4.1]{seppalainen00}. We include a proof, which is slightly different than the one in \cite{seppalainen00}, to save the reader the trouble of translating notations between the two papers.
\begin{lmm}\label{rhotail}
Recall the law $\lambda$ defined in Section \ref{proofsec}. For any $x$, 
\[
\lambda(\{P\in \Gamma'': \rho(x,P) > C_1T\}) \le C_2 e^{-C_3 T},
\]
where $C_1$, $C_2$, and $C_3$ are positive constants that depend only on $d$.
\end{lmm}
\begin{proof}
Let us fix $x$, and write $S$ instead of $S(x,P)$, etc. Let $P$ be a random element of $\Gamma''$ with law $\lambda$. As in the proof of Lemma \ref{gprob}, extend each~$P_x$ to a Poisson process $P_x'$ on $(-\infty, T]$.  Let $\cq$ be the set of all infinite lazy random walk paths on $\zz^d$ starting from $x$. A typical element of $\cq$ is an infinite sequence $q=(q_0,q_1\ldots)$ of elements of $\zz^d$ such that $q_0=x$ and $|q_{i+1}-q_{i}|_1\le 1$ for each $i\ge 0$, where $|\cdot|_1$ denotes $\ell^1$ norm. To any $q\in \cq$, let us associate a sequence of times $\tau_0(q),\tau_1(q),\cdots \in (-\infty, T]$ as follows. Start with $\tau_0(q)=T$. Having defined $\tau_i(q)$, let $\tau_{i+1}(q)$ be the largest $t\in (-\infty, \tau_i(q))$ such that  the clock at site $q_i$ rings at time $t$. Then, by the memoryless property of Poisson processes, the sequence $\tau_1(q), \tau_2(q),\ldots$ is a Poisson process on $(-\infty, T]$ with intensity $1$. 

We claim that for each $y\in S$, there is some $q\in \cq$ and some $i$ such that $q_i = y$ and $\tau_i(q) \ge 0$. To show this, we prove by induction on $k$ that for each $k$, and for each $y\in S_k$, there is some $q\in \cq$ and some $i$ such that $q_i = y$ and  $\tau_i(q)\ge t_k$. Clearly, this is true for $k=0$. Suppose that it holds for some $k$. Recall that $S_{k+1}$ is obtained by taking the union of $S_k$ and the set of neighbors of $y$, where $y$ is the element of $S_k$ where the clock rings at time $t_{k+1}$. Take any $z\in S_{k+1}\setminus S_k$, so that $z$ is a neighbor of $y$. By the induction hypothesis, there exists a lazy random walk path $q_0,\ldots,q_i$ starting at at $x$ and ending at $y$, and times $\tau_0(q),\ldots, \tau_i(q)$ defined as above, such that $\tau_i(q)\ge t_k$. Let $j$ be the number of times that the clock at $y$ rings in $[t_{k+1}, \tau_i(q))$. Then $j\ge1$, since we know that the clock at $y$ rings at time $t_{k+1}$. Define $q_{i+1}=q_{i+2}=\cdots=q_{i+j-1} = y$, and $q_{i+j} = z$. Then it is clear that $\tau_{i+1}(q)> \cdots> \tau_{i+j}(q)=t_{k+1}$. This proves that the claimed property holds for every $z\in S_{k+1}\setminus S_k$, and it holds in $S_{k+1}\cap S_k =S_k$ by the induction hypothesis and the fact that $t_k\ge t_{k+1}$. This completes the proof of the induction step.

Fix some $a\in (0,1)$, to be chosen later. Take any $y\in \zz^d$ such that $|y-x|_1 \ge T/a$, and any $q\in \cq$ and $i\ge 0$ such that $y=q_i$. Then $i\ge |y-x|_1\ge T/a$. Since $T-\tau_i(q)$ has a gamma distribution with shape parameter $i$ and scale parameter $1$, Lemma \ref{gammalmm} gives 
\begin{align*}
\pp(\tau_i(q)\ge 0) &= \pp(T-\tau_i(q) \le T) \le e^{(1-(T/i)+\log (T/i))i}\le e^{(1-a+\log a)i},
\end{align*}
where the last inequality holds because $x\mapsto 1-x+\log x$ is an increasing function in $(0,1)$ and $T/i \le a$. Since there are $(2d+1)^i$ paths of length $i$ starting from $x$, this shows that
\begin{align*}
\pp(y\in S) &\le \sum_{q,i\, : \, q_i =y} \pp(\tau_i(q)\ge 0)\\
&\le \sum_{i\ge |y-x|_1} (2d+1)^i e^{(1-a+\log a)i}. 
\end{align*}
Choosing $a$ so small that $1-a+\log a < -4d-2$, we get 
\begin{align}\label{pseq}
\pp(y\in S) \le C_1 e^{-C_2 |y-x|_1},
\end{align}
where $C_1$ and $C_2$ depend only on $d$. Note that this holds only when $|y-x|_1\ge T/a$, and $a$ depends only on $d$. Summing over all such $y$ completes the proof. 
\end{proof}
Still fixing $T$, recall the map $\Psi$ defined in Section \ref{proofsec}, which maps a surface $f$ and a collection of update times $P$ to the updated surface $\Psi(f,P)$ at time $T$. Given a finite set $D\subseteq \zz^d$, define $\Psi_D$ to be the analogous map where only the heights at $D$ are updated. 
\begin{lmm}\label{slmm}
Take any $f\in \Omega$ and $P\in \Gamma''$. Then for any finite $D$ and any $x$ such that $S(x,P)\subseteq D$, $\Psi_D(f,P)$ and $\Psi(f,P)$ are the same at $x$. 
\end{lmm}
\begin{proof}
Let us write $S$ instead of $S(x,P)$, etc. Consider the two update processes --- the first, where all sites are updated, and the second, where only sites in $D$ are updated. In the time interval $[0,t_K]$, no site in $S = S_K$ is updated in either process. Instantly after time $t_K$, a site in $S_{K-1}$ is updated. Since this update depends only on the heights in $S_K$, the updated height is the same for the two processes. Continuing inductively like this, we deduce that the final heights at $S_0 = \{x\}$ are the same for the two processes. 
\end{proof}
Take any $N$, and consider the ballistic deposition process $h_N(t,\cdot)$ starting with $h_N(0,\cdot)\equiv 0$, and growing by allowing updates only in $B_N$. Our next lemma shows that for $x$ fixed and $N$ large, $h_N(t,x)$ is equal to $h(t,x)$ with high probability if $1\ll t\ll N$.
\begin{lmm}\label{hzero}
Let $\{t_N\}_{N\ge 1}$ be a sequences of positive real numbers such that $1\ll t_N \ll N$ as $N\to\infty$. Then for any $x$, $h_N(t_N,x)- h(t_N, x) \to 0$ in probability as $N\to\infty$.
\end{lmm}
\begin{proof}
Recall the quantity $\rho(x,P)$ defined above. Let $\rho_N$ be the value of $\rho(x,P)$ and $S_N$ be the set $S(x,P)$ when $T=t_N$ and $P$ is a random element of $\Gamma''$ with law $\lambda$ (see Section \ref{proofsec}).  By Lemma \ref{rhotail}, $\pp(\rho_N > C_1 t_N) \le C_2 e^{-C_3 t_N}$, where $C_1$, $C_2$ and $C_3$ are positive constants that depend only on $d$. But if $\rho_N \le C_1 t_N$, and if $N$ is large enough, then $S_N \subseteq B_N$, because $t_N \ll N$ as $N\to \infty$. And if $S_N \subseteq B_N$, then by Lemma \ref{slmm}, $h_N(t_N, x) = h(t_N,x)$. Since $t_N \gg 1$ as $N\to \infty$, this proves the claim.
\end{proof}
The next lemma says the same thing as the previous one, except that now $t_N$ is randomly chosen from an exponential distribution. 
\begin{lmm}\label{expconv}
Let $\{a_N\}_{N\ge 1}$ be a sequence of positive real numbers such that $1\ll a_N \ll N$ as $N\to \infty$. For each $N$, let $t_N$ be an exponential random variable with mean $a_N$. Then for any $x$, $h_N(t_N,x)- h(t_N, x) \to 0$ in probability as $N\to\infty$.
\end{lmm}
\begin{proof}
Note that
\begin{align*}
\pp(h_N(t_N,x)\ne  h(t_N, x)) &= \int_0^\infty a_N^{-1} e^{-t/a_N} \pp(h_N(t,x)\ne  h(t, x)) dt\\
&= \int_0^\infty e^{-s}\pp( h_N(a_Ns, x) \ne h(a_Ns, x))ds.
\end{align*}
By Lemma \ref{hzero}, $\pp( h_N(a_Ns, x) \ne h(a_Ns, x))\to 0$ as $N\to \infty$ for every $s>0$. By the dominated convergence theorem, this completes the proof. 
\end{proof}
The above lemmas show that for fixed $x$, $h_N(t,x) = h(t,x)$ with high probability if $1\ll t\ll N$, where $t$ may be fixed or random. The next lemma shows that this suffices for deducing that $h_N(t,\cdot)$ is close to $h(t, \cdot)$ on $\Omega$.
\begin{lmm}\label{problmm}
Let $\{f_N\}_{N\ge 1}$ be a sequence of random elements of $\Omega$. Then $f_N\to 0$ in probability if and only if $f_N(x)\to 0$ in probability for each $x$. 
\end{lmm}
\begin{proof}
It is easy to check that the following is a metric for the product topology on~$\Omega$:
\[
d(f,g) := \sum_{x\in \zz^d} 2^{-|x|} \frac{|f(x)-g(x)|}{1+|f(x)-g(x)|}. 
\]
Suppose that $f_N(x) \to 0$ in probability for each $x$. Take any $\ve >0$. Find $R>0$ so large that
\[
\sum_{x\, : \, |x|>R} 2^{-|x|} \le \frac{\ve}{2}. 
\] 
Note that 
\begin{align*}
\pp(d(f_N, 0) > \ve) &\le \pp\biggl(\sum_{x\, : \, |x|\le R} 2^{-|x|} \frac{|f_N(x)|}{1+|f_N(x)|} > \frac{\ve}{2} \biggr) \\
&\qquad \qquad + \pp\biggl(\sum_{x\, : \, |x|> R} 2^{-|x|} \frac{|f_N(x)|}{1+|f_N(x)|} > \frac{\ve}{2} \biggr).
\end{align*}
But, by the choice of $R$, the second event on the right is impossible. Thus,
\begin{align*}
\pp(d(f_N, 0) > \ve) &\le \sum_{x\, : \, |x|\le R} \pp\biggl(2^{-|x|} \frac{|f_N(x)|}{1+|f_N(x)|} > \frac{\ve}{2L}\biggl),
\end{align*}
where $L$ is the number of $x$ such that $|x|\le R$. Since $R$ and $L$ are fixed, and $f_N(x)\to 0$ in probability as $N\to \infty$, the sum on the right tends to zero as $N\to \infty$. Thus, $f_N \to 0$ in probability. The proof of the converse implication is similar (and easier).
\end{proof}
Combining Lemma \ref{tightnew}, Lemma \ref{expconv} and Lemma \ref{problmm}, we arrive at the following corollary.
\begin{cor}\label{newcor}
Let $\{a_N\}_{N\ge 1}$ be a sequence of positive real numbers such that $1\ll a_N \ll N$ as $N\to \infty$. For each $N$, let $t_N$ be an exponential random variable with mean $a_N$. For each $t$, let $w_N(t,\cdot) := h_N(t,\cdot)-h_N(t,0)$. Then $\{w_N(t_N,\cdot)\}_{N\ge 1}$ is a tight family of $\Omega$-valued random variables, and any subequential weak limit of the laws of these random variables is a stationary probability measure for the Markov semigroup $\{\cp_t\}_{t\ge 0}$.
\end{cor}
\begin{proof}
By Lemma \ref{expconv} and Lemma \ref{problmm}, $h_N(t_N,\cdot)-h(t_N,\cdot)\to 0$ in probability on $\Omega$. Moreover, $h_N(t_N,0) - h(t_N,0)$ also tends to zero in probability. Thus, $w_N(t_N,\cdot)-w(t_N, \cdot)\to 0$ in probability. By Lemma \ref{tightnew}, $\{w(t_N,\cdot)\}_{N\ge 1}$ is a tight family, and the law of any subsequential weak limit is a stationary probability measure for  $\{\cp_t\}_{t\ge 0}$. Therefore, these properties also hold for the sequence $\{w_N(t_N,\cdot)\}_{N\ge 1}$. 
\end{proof}
We are now ready to complete the proof of Theorem \ref{simulthm}. There are two key observations. First, note that if we only update in the finite set $B_N$, then the updates can be ordered as an increasing sequence, and this ordered sequence behaves  just as if we are picking vertices uniformly at random from $B_N$ at each turn, and updating the height at the chosen vertex. In other words, we get the discrete update process of Theorem \ref{simulthm}. Moreover, by time $t$, the total number of updates in $B_N$ has a Poisson distribution with mean $|B_N| t$. Thus, if a time $T$ is chosen at random from the exponential distribution with mean $a_N$, and $K$ is the number of updates in $B_N$ within time $t$, then for any $j\ge 0$, 
\begin{align*}
\pp(K=j)  &= \int_0^\infty a_N^{-1}e^{-t/a_N} \pp(K=j\mid T=t) dt\\
&=  \int_0^\infty a_N^{-1}e^{-t/a_N} e^{-|B_N|t}\frac{(|B_N|t)^j}{j!} dt\\
&= \frac{(|B_N|a_N)^j}{(|B_N|a_N+1)^{j+1}},
\end{align*}
where the last step is computed using the gamma integral. 
Thus, $K$ has a geometric distribution with success probability $1/(|B_N|a_N+1)$. If we equate this to the quantity $p(N)$ from the statement of Theorem \ref{simulthm}, then the condition $1\ll a_N \ll N$ needed for applying Corollary \ref{newcor} holds if and only if $N^{-d-1}\ll p(N)\ll N^{-d}$. This completes the proof of Theorem \ref{simulthm}.

\section{Proof of Theorem \ref{fellerthm}}\label{fellerproofsec}
Recall all notations from Section \ref{proofsec}. We will first show that the semigroup $\{\cp_t\}_{t\ge 0}$ has the Feller property. Take any bounded continuous function $\Phi:\Omega \to \rr$, and any $T>0$. Our objective is to show that $\cp_T \Phi$ is also a bounded continuous function. It is obviously bounded, so the task is to show that it is continuous.

Let $P$ be drawn from the probability measure $\lambda$. Take any $f\in \Omega$. Recall that $\Psi'(f,P)$ is the value of $w(T,\cdot)$ if we start from initial configuration $w(0,\cdot) = f$ and apply the deposition process $P$ to obtain the height function at time $T$. Take a sequence $f_n \to f$. Since $P\in \Gamma''$ with probability one, and $\Psi'$ is continuous on $\Omega \times \Gamma''$ (by Lemma \ref{contlmm} and the definition of $\Psi'$), we have that $\Psi'(f_n,P)\to \Psi'(f,P)$ almost surely as $n\to\infty$. Thus, since $\Phi$ is continuous, the bounded convergence theorem shows that
\begin{align*}
\cp_T\Phi(f_n) &= \ee(\Phi(\Psi'(f_n,P))) \to \ee(\Phi(\Psi'(f,P))) = \cp_T \Phi(f)
\end{align*}
as $n\to \infty$. This proves the Feller property of $\{\cp_t\}_{t\ge 0}$.

Next, we show that $\{\cp_t\}_{t\ge 0}$ does not have the strong Feller property. For that, we have to exhibit a bounded measurable function $\Phi:\Omega \to \rr$, and some $T>0$, such that $\cp_T \Phi$ is not continuous. We will, in fact, produce a $\Phi$ such that $\cp_T \Phi$ is not continuous for any $T> 0$. By the Feller property, it suffices to show this for all $T\ge 1$. In this proof, $C_1,C_2,\ldots$ will denote arbitrary positive constants that depend only on $d$, whose values may change from line to line. 

Let $A$ be the set of all $f\in \Omega$ such that 
\[
\limsup_{u\to\infty} \frac{f(ue_1)}{u} \ge 1.
\]
It is easy to see that $A$ is a measurable subset of $\Omega$. Let $\Phi := 1_A$. Take any $T\ge 1$. We will now show that $\cp_T \Phi$ is not continuous at $f\equiv 0$. To prove this, we will first show that $\cp_T \Phi(f) = 0$, and then exhibit a sequence $f_n \to f$ such that $\cp_T \Phi(f_n) =1$ for each $n$. We need the following simple lemma.
\begin{lmm}\label{poissonlmm}
Take any $T\ge 1$, as above. If $X$ is a \textup{Poisson}$(T)$ random variable, then for any $x>0$,
\[
\pp(X \ge x) \le Ce^{-x/T},
\]
where $C$ does not depend on $T$. 
\end{lmm}
\begin{proof}
Note that 
\begin{align*}
\pp(X\ge x) &\le e^{-x/T} \ee(e^{X/T}) \\
&= e^{-x/T} \sum_{k=0}^\infty e^{-T}\frac{(e^{1/T} T)^k}{k!} = e^{-x/T} e^{T(e^{1/T} - 1)}.
\end{align*}
To complete the proof, note that $T(e^{1/T}-1)$ can be bounded by a constant that does depend on $T$, since $T\ge 1$. 
\end{proof}

Let all notations be as in Sections \ref{proofsec} and \ref{simulproof}. Draw $P$ from $\lambda$, and let $g := \Psi(f,P)$ and $g' := \Psi'(f,P)$. Take any $x\in \zz^d$. By equation \eqref{pseq}, we know that there are positive constants $C_1$, $C_2$, and $a$ such that if $|y-x|_1\ge T/a$, then $\pp(y\in S) \le C_1 e^{-C_2|y-x|_1}$, where $S = S(x,P)$.  Letting $\rho(x,P)$ be as in Lemma \ref{rhotail}, this estimate shows that if $R \ge T/a$, then 
\begin{align}
\pp(\rho(x,P) \ge R) &\le \sum_{y\, :\, |y-x|_1 \ge R} \pp(y\in S)\notag\\
&\le  \sum_{y\, :\, |y-x|_1 \ge R} C_1 e^{-C_2|y-x|_1} \le C_3 e^{-C_4R}. \label{rhobound}
\end{align}
Now, from the proof of Lemma \ref{slmm}, it is easy to see that 
\begin{align}\label{gxpy}
g(x)\le \sum_{y\in S} |P_y|.
\end{align}
Take any $R\ge T/a$, and let $B(x,R)$ denote the $\ell^1$ ball of radius $R$ centered at $x$. The above inequality shows that if $g(x) \ge R^{d+1}$ and $S\subseteq B(x,R)$, then for some $y\in B(x,R)$, 
\[
|P_y| \ge \frac{R^{d+1}}{|B(x,R)|} \ge C_1 R. 
\]
Thus, by the fact that $|P_y|$ is a Poisson$(T)$ random variable for each $y$, and by Lemma~\ref{poissonlmm}, we get
\begin{align*}
\pp(g(x) \ge R^{d+1}) &\le \sum_{y\in B(x,R)} \pp(|P_y|\ge C_1 R) + \pp(\rho(x,P) > R)\\
&\le C_1 R^d e^{-C_2 R/T} + C_3 e^{-C_4 R}\le C_5 T^de^{-C_6R/T},
\end{align*}
where the last inequality holds because $x^d e^{-x} \le C_1e^{-C_2x}$. 
Since $g'(x)=g(x)-g(0)\le g(x)$, this shows that if $|x|_1^{1/(2d+2)} \ge T/a$, then 
\begin{align*}
\pp(g'(x)\ge \sqrt{|x|_1}) &\le C_1 T^de^{-C_2|x|_1^{1/(2d+2)} /T}.
\end{align*}
Consequently, if $R$ is so large that $R^{1/(2d+2)} \ge T/a$ and $\sqrt{R} \le R/2$, then 
\begin{align*}
\pp\biggl(\sup_{x\, : \, |x|_1\ge R} \frac{g'(x)}{|x|_1} \ge \frac{2}{3}\biggr) &\le \sum_{x\, : \, |x|_1\ge R}\pp(g'(x) \ge \sqrt{|x|_1})\\
&\le \sum_{k=0}^\infty \sum_{x\,: \, 2^kR\le |x|_1< 2^{k+1}R}C_1 T^de^{-C_2|x|_1^{1/(2d+2)} /T} \\
&\le C_1 T^dR^{d} e^{-C_2R^{1/(2d+2)}/T}.
\end{align*}
This shows that
\begin{align*}
0\le \cp_T \Phi(f) &= \pp(g' \in A) \\
&\le \lim_{R\to\infty} \pp\biggl(\sup_{x\, : \, |x|_1\ge R} \frac{g'(x)}{|x|_1} \ge \frac{2}{3}\biggr) = 0.
\end{align*}
Next, take any positive integer $n$. For $x = (x_1,\ldots,x_d)\in \zz^d$, define 
\[
f_n(x) := 
\begin{cases}
x_1  &\text{ if } x> n,\\
0 &\text{ otherwise.}
\end{cases}
\]
Define $g_n := \Psi(f_n,P)$ and $g_n' := \Psi'(f_n, P)$. Just like \eqref{gxpy}, we now have
\[
\min_{y\in S} f_n(y) \le g_n(x) \le \sum_{y\in S} |P_y| + \max_{y\in S} f_n(y).
\]
Consequently,
\begin{align*}
g_n'(x) &= g_n(x)-g_n(0) \\
&\ge \min_{y\in S(x,P)} f_n(y) - \sum_{y\in S(0, P)} |P_y| - \max_{y\in S(0,P)} f_n(y).
\end{align*}
By a similar argument as above, using \eqref{rhobound}, it follows that the second and third terms on the right are finite almost surely. Thus, we have that with probability one,
\[
\limsup_{u\to\infty } \frac{g_n'(ue_1)}{u} = \limsup_{n\to\infty} \frac{\min_{y\in S(ue_1,P)} f_n(y) }{u}.
\]
For any $u\ge n$ and $y\in S(ue_1,P)$, we have
\begin{align*}
f_n(y) &\ge y_1 - n \ge u - \rho(ue_1,P) - n. 
\end{align*} 
Again, using the tail bound \eqref{rhobound}, it is not hard to show that with probability one, 
\[
\lim_{u\to\infty} \frac{\rho(ue_1,P)}{u} = 0. 
\]
Combining all of the above, we see that $g_n'\in A$ with probability one. Thus, $\cp_T \Phi(f_n) = 1$ for all $n$. But $f_n \to f$ in the product topology on $\Omega$, and as shown above, $\cp_T \Phi(f)=0$. This shows that $\cp_T \Phi$ is not a continuous map.

\section{Proof of Theorem \ref{uithm}}
As before, $C_1,C_2,\ldots$ will denote positive constants that depend only on $d$, whose values may change from line to line. Consider the initial conditions $h(0, x) = x_1$ and $h'(0, x) = x_1 - 1$. Let us couple two ballistic deposition processes with these two initial conditions using the same deposition events. Then $h(t,x) - h'(t,x)=1$ for all $t$ and $x$, because $h(0,x) = h'(0,x)+1$ for all $x$. But it is also clear that $\{h(t,x)\}_{t\ge 0, x\in \zz^d}$ has the same law as $\{h'(t,x+e_1)\}_{t\ge 0, x\in \zz^d}$, because $h(0,x) = h'(0,x+e_1)$ for all $x$. Thus,
\begin{align}
&\ee(h(t,x+e_1)) - \ee(h(t,x)) \notag \\
&= \ee(h'(t,x+e_1)+1) - \ee(h'(t,x+e_1)) = 1. \label{hvalue}
\end{align}
By a similar logic, 
\begin{align}\label{hvalue2}
\ee(h(t, x+e_i)) = \ee(h(t,x)) \ \text{ for any $i\ne 1$.}
\end{align}
Next, we claim that 
\begin{align}\label{toshow}
\ee(h(t,0))\le Ct
\end{align}
for all $t\ge 1$, where $C$ depends only on $d$. To prove this, take any real number $T\ge 1$ and nonnegative  integer $R$. Define two new processes, $h_R$ and $g$, which use the same updates as $h$, but starting from initial conditions $h_R(0,\cdot)$ and $g(0,\cdot)$, defined as $g(0,\cdot)\equiv0$ and 
\[
h_R(0,x) =
\begin{cases}
R &\text{ if } x_1\le R,\\
x_1 &\text{ if } x_1 > R.
\end{cases}
\]
Let $P$, $S = S(0,P)$, and $\rho = \rho(0,P)$ be as in Section \ref{simulproof}. Then by Lemma \ref{slmm} and monotonicity, it is easy to see that
\[
h(T,0) \le h_\rho(T,0) = g(T,0) + \rho.
\]
It follows from~\eqref{rhobound} that $\ee(\rho)\le C_1 T$, and by Lemma \ref{penroselmm}, $\ee(g(T,0))\le C_2 T$.  Using these in the above display, we get \eqref{toshow}.

Armed with \eqref{toshow}, it is now easy to show that Lemmas \ref{derivlmm} and \ref{integratelmm} hold for the initial state $h(0,x)=x_1$. As for Lemma \ref{bdlmm}, the first equation in the proof of that lemma holds for the initial state $h(0,x) =x_1$ because of equations~\eqref{hvalue} and~\eqref{hvalue2}. The proof of the lemma can be completed as before, using \eqref{toshow} in the last step. This completes the proof of \eqref{newavg}.  

Next, it is easy to see how the proof of Lemma \ref{maintight} also goes through, because the gradient process  under the initial condition $h(0,x)=x_1$ is translation invariant. From this, we can proceed as in the proof of Theorem \ref{mainthm} to show that if $U(T)$ is chosen uniformly at random from $[0,T]$, then the family of random surfaces $\{w(U(T),\cdot)\}_{T\ge 0}$ is tight, and the law of any subsequential weak limit is an invariant probability measure for our Markov process. If, now, the condition~\eqref{newavg2} holds, then it means that the family of random variables $\{w(U(T), e_1)\}_{T\ge 0}$ is uniformly integrable. But by \eqref{hvalue}, $\ee(w(U(T), e_1)) = 1$ for any $T$. Thus, if $q$ is a random surface drawn from an invariant measure obtained as above, then $\ee(q(e_1))=1$. The gradient field for a surface drawn from this invariant measure cannot be invariant under lattice symmetries (specifically, under the map $x \to -x$) like the one obtained in Theorem \ref{mainthm}. Thus, our Markov  process has more than one invariant measure if~\eqref{newavg2} is valid.

\section*{Acknowledgements}
The author is grateful to the anonymous referees for a number of useful suggestions and references.

\bibliographystyle{plainnat}
\bibliography{myrefs}

\end{document}